\documentclass{amsart}

\usepackage{array}
\usepackage{enumerate, url}

\newcommand{\CC}{\mathbb{C}}
\newcommand{\EE}{\mathbb{E}}
\newcommand{\NN}{\mathbb{N}}
\newcommand{\PP}{\mathbb{P}}
\newcommand{\ZZ}{\mathbb{Z}}
\newcommand{\del}{\partial}

\renewcommand{\bold}[1]{\mathbf{#1}}

\begin{document}

\title[Linr. Trans. \& The Multv. Gen. Fn.]{Linear Transformations \& the Multivariate Generating Function}
\author[M. C. Burkhart]{Michael C. Burkhart}



\theoremstyle{plain}
\newtheorem{thm}{Theorem}
\newtheorem*{lem}{Lemma}
\newtheorem{prop}{Proposition}
\newtheorem*{cor}{Corollary}
\newtheorem{methd}{Method}

\theoremstyle{definition}
\newtheorem*{defn}{Definition}
\newtheorem*{conj}{Conjecture}
\newtheorem*{exmp}{Example}

\theoremstyle{remark}
\newtheorem*{rem}{Remark}
\newtheorem*{note}{Note}
\newtheorem*{case}{Case}


\begin{abstract}
This note examines linear combinations of multi-indexed sequences and derives the multivariate generating function of such a linear combination in terms of the original sequence's m.g.f.  Applications include finding distributions and moments of non-negative discrete random variables conditioned on non-negative linear combinations of the original variables.  Examples include independent Poisson r.v.'s and a $d$-variate multinomial distribution.
\end{abstract}


\maketitle


\section{Introduction}
Set $\NN = \ZZ_{\ge 0}$.  Let $\frak b: \NN^d \rightarrow \CC $ denote a multi-indexed sequence of complex numbers.  Then the multivariate generating function for $\frak b$ is given by:
\[
G_{\frak b}(t_1,\dotsc, t_d) 
= \sum_{(j_1,\dotsc, j_d) \in\NN^d} \frak b_{(j_1,\dotsc,j_d)} t_1^{j_1}\dotsb t_d^{j_d}
\]
Such generating functions have applications throughout discrete mathematics, especially to combinatorial classes and probability distributions~\cite{FS}, and prove useful in finding recurrences, moments, and asymptotics~\cite{Wilf, PW}.  Multivariate generating functions can be used to obtain conditional distributions~\cite{X, JKB}.  Applications to biochemistry include stochastic models of chemical network theory, in particular chemical kinetics~\cite{SZ}.


\section{Main Result}
Fix some matrix $\bold A=(a_{ij}) \in \bold{Mat}_{m\times d}(\NN)$.  Define a new multi-indexed sequence $\frak c: \NN^m \rightarrow \CC$ by taking linear combinations of the sequence $\frak b$ using the coefficients of the matrix $\bold A$.  That is, for $(k_1,\dotsc,k_m) \in \NN^m$ set:
\[
\frak c_{(k_1,\dotsc,k_m)} 
= \sum_{\substack{(j_1,\dotsc,j_d)\in\NN^d \\ (k_1,\dotsc,k_m)^T=\bold A(j_1,\dotsc,j_d)^T }} \frak b_{(j_1,\dotsc,j_d)}
\]

\begin{thm}\label{main}
The analogously-defined multivariate generating function for the sequence $\frak c$ is given by:
\[
G_{\frak c}(z_1,\dotsc, z_m) 
=  G_{\frak b} (\Pi_{i=1}^m z_i^{a_{i1}}, \dotsc, \Pi_{i=1}^m z_i^{a_{id}} )
\]
where $\bold A = (a_{ij})$.
\end{thm}

\begin{proof}
This proof mirrors Sontag and Zeilberger's proof of the special case where $\frak b_{(j_1,\dotsc,j_d)}$ is the joint probability distribution for independent Poisson random variables~\cite{SZ}.  From the definition:
\begin{multline*}
G_{\frak c} (z_1,\dotsc, z_m)
= \sum_{(k_1,\dotsc, k_m) \in\NN^m} \frak c_{(k_1,\dotsc,k_m)} z_1^{k_1}\dotsb z_m^{k_m} \\
= \sum_{(k_1,\dotsc, k_m) \in\NN^m} \Biggl( \sum_{\substack{(j_1,\dotsc,j_d)\in\NN^d \\ (k_1,\dotsc,k_m)^T=\bold A(j_1,\dotsc,j_d)^T }} \frak b_{(j_1,\dotsc,j_d)}\Biggr) z_1^{k_1}\dotsb z_m^{k_m}
\end{multline*}
Exchaning the order of summation then yields:
\begin{multline*}
G_{\frak c} (z_1,\dotsc, z_m)
= \sum_{(j_1,\dotsc,j_d)\in\NN^d} \Biggl( \sum_{\substack{(k_1,\dotsc, k_m) \in\NN^m \\ (k_1,\dotsc,k_m)^T=\bold A(j_1,\dotsc,j_d)^T }} \frak b_{(j_1,\dotsc,j_d)} z_1^{k_1}\dotsb z_m^{k_m} \Biggr) \\
= \sum_{(j_1,\dotsc,j_d)\in\NN^d} \frak b_{(j_1,\dotsc,j_d)} z_1^{a_{11}j_1+\dotsb+a_{1d}j_d}\dotsb z_m^{a_{m1}j_1+\dotsb+a_{md}j_d} \\
= \sum_{(j_1,\dotsc,j_d)\in\NN^d} \frak b_{(j_1,\dotsc,j_d)} (z_1^{a_{11}}\dotsb z_m^{a_{m1}})^{j_1} \dotsb (z_1^{a_{1d}} \dotsb z_m^{a_{md}})^{j_d} \\
=G_{\frak b}(z_1^{a_{11}}\dotsb z_m^{a_{m1}}, \dotsc, z_1^{a_{1d}}\dotsb z_m^{a_{md}})
=G_{\frak b} (\Pi_{i=1}^m z_i^{a_{i1}}, \dotsc, \Pi_{i=1}^m z_i^{a_{id}} )
\end{multline*}

\end{proof}


\section{Probability Generating Functions}

Let $X_1,\dotsc,X_d$ be non-negative discrete random variables (not necessarily independent).  The multivariate probability generating function (henceforth denoted p.g.f.) of $X_1,\dotsc,X_d$ is then given by:
\[
G_{\bold X}(t_1,\dotsc, t_d) 
= \sum_{(j_1,\dotsc, j_d) \in\NN^d} \PP(X_1=j_1,\dotsc, X_d=j_d)\ t_1^{j_1}\dotsb t_d^{j_d}
\]
Define new random variables $Y_1, \dotsc, Y_m$ by taking linear combinations of the $X_i$:
\[\tag{$\star$}
\left(\begin{smallmatrix} Y_1 \\ \vdots \\ Y_m \end{smallmatrix} \right)
= \bold A \left(\begin{smallmatrix} X_1 \\ \vdots \\ X_d \end{smallmatrix} \right)
\]
\begin{prop}\label{p.g.f.}
The analogously-defined multivariate p.g.f. for $Y_1,\dotsc, Y_m$ is given by:
\begin{multline*}
G_{\bold Y}(z_1,\dotsc,z_m) 
=\sum_{(k_1,\dotsc, k_m) \in\NN^m} \PP(Y_1=k_1,\dotsc, Y_m=k_m)\ z_1^{k_1}\dotsb z_m^{k_m} \\
= G_{\bold X}(\Pi_{i=1}^m z_i^{a_{i1}}, \dotsc, \Pi_{i=1}^m z_i^{a_{id}} )
\end{multline*}
\end{prop}
Apply the theorem to the multi-indexed sequences $\frak b_{(j_1,\dotsc,j_d)}=\PP(X_1=j_1,\dotsc, X_d=j_d)$ and $\frak c_{(k_1,\dotsc,k_m)}=\PP(Y_1=k_1,\dotsc, Y_m=k_m)$.

\subsection*{Independence}  
When in addition $X_1,\dotsc, X_d$ are independent non-negative discrete random variables, the p.g.f. takes the form:
\[
G_{\bold X}(t_1,\dotsc,t_d)
=\prod_{r=1}^d \Bigl( \sum_{j_r\in\NN} \PP(X_r=j_r)\ t_r^{j_r} \Bigr)
=\prod_{r=1}^d G_{X_r}(t_r)
\]
where $G_{X_r}$ is the single-variable probability generating function for $X_r$.  It follows that the p.g.f. for the linear combinations $Y_1,\dotsc, Y_m$ is given by:
\[
G_{\bold Y}(z_1,\dotsc,z_m)
=\prod_{r=1}^d \Bigl( G_{X_r}(\Pi_{i=1}^m z_i^{a_{ir}}) \Bigr)
\]
\begin{exmp}
When $X_1,\dotsc, X_d$ are independent Poisson random variables (\emph{cf.}~\cite{SZ}), $X_r\sim Poisson(\lambda_r)$, $1\leq r\leq n$:
\[
G_{\bold Y}(z_1,\dotsc, z_m) 
= \exp\left( \sum_{r=1}^d \lambda_r \Bigl( \prod_{i=1}^m z_i^{a_{ir}} -1 \Bigr) \right)
\]
\end{exmp}
Simply note that, for $1\leq r\leq n$:
\[
G_{X_r}(t_r) 
= \sum_{i\ge 0} \frac{\lambda_r^i e^{-\lambda_r}}{i!} t_r^{i}
= e^{-\lambda_r} \sum_{i\ge 0} \frac{(\lambda_r t_r)^i}{i!}
= \exp(\lambda_r t_r - \lambda_r)
\]


\section*{The Conditional Distribution of $\bold X \mid \bold Y$}

\begin{prop}\label{multi p.g.f.}
Let the random vectors $\bold X, \bold Y$ be given as in $(\star)$.  Then the joint multivariate p.g.f. for $\bold X, \bold Y$ is given:
\begin{multline*}
G_{\bold X, \bold Y} (t_1,\dotsc,t_d;z_1,\dotsc, z_m) \\
=\sum_{\substack{(j_1,\dotsc,j_d;k_1,\dotsc,k_m)\in\NN^{d+m} \\ (k_1,\dotsc,k_m)^T=\bold A(j_1,\dotsc,j_d)^T }}
\PP(\bold X = \bold j, \bold Y=\bold k)\ t_1^{d_1}\dotsb t_d^{j_d} \cdot z_1^{k_1}\dotsb z_m^{k_m} \\
= G_{\bold X}(t_1\cdot \Pi_{i=1}^m z_i^{a_{i1}}, \dotsc, t_d\cdot \Pi_{i=1}^m z_i^{a_{id}} )
\end{multline*}
\end{prop}
This follows from noting that $\PP(\bold X = \bold j, \bold Y=\bold k)=\PP(\bold X = \bold j)$ and then applying the theorem with $\frak b_{(j_1,\dotsc,j_d)}=\PP(\bold X = \bold j)\  t_1^{d_1}\dotsb t_d^{j_d}$.  From joint multivariate p.g.f., it is possible to obtain the conditional p.g.f. of $\bold X$ given $Y_1=k_1,\dotsc, Y_m=k_m$~\cite{X}.  Further, the conditional pure (resp. mixed) factorial momements correspond to taking the coefficient of $z_1^{k_1}\dotsb z_m^{k_m}$ in $G_{\bold X, \bold Y}$ (which will be a polynomial in $t_1,\dotsc,t_d$), taking pure (resp. mixed) partial derivatives with respect to the $t_r$, evaluating at $t_1=\dotsb=t_d=1$, and dividing by $\PP(\bold Y=\bold k)$.  Let $[z_1^{j_1}\dotsb z_d^{j_d}] G(z_1,\dotsc,z_d)$ denote the process of extracting the coefficient of $z_1^{j_1}\dotsb z_d^{j_d}$ in the formal power series $G(z_1,\dotsc,z_d)=\sum_{(j_1,\dotsc,j_d)\in\NN^d} G_{(j_1,\dotsc,j_d)}z_1^{j_1}\dotsb z_d^{j_d}$.  With this notation, it follows that:
\begin{multline*}
\EE(X_1^{(s_1)}\dotsb X_d^{(s_d)} \mid Y_1=k_1, \dotsc, Y_m=k_m)  \\
:= \EE\left(\frac{X_1!}{(X_1-s_1)!} \dotsb \frac{X_d!}{(X_d-s_d)!}\ \Big\vert\ Y_1=k_1, \dotsc, Y_m=k_m \right)  \\
= \sum_{\substack{(j_1,\dotsc, j_d) \in\NN^d \\ (k_1,\dotsc,k_m)^T=\bold A(j_1,\dotsc,j_d)^T }} \frac{j_1!}{(j_1-s_1)!} \dotsb \frac{j_d!}{(j_d-s_d)!}\ \PP(X_1=j_1,\dotsc, X_d=j_d \mid \bold Y) \\
=\frac{[z_1^{k_1}\dotsb z_m^{k_m}] \left( \frac{\del^{s_1}}{\del t_1^{s_1}} \dotsb \frac{\del^{s_d}}{\del t_d^{s_d}} G_{\bold X, \bold Y} (1,\dotsc,1;z_1,\dotsc, z_m) \right)}{[z_1^{k_1}\dotsb z_m^{k_m}] G_{\bold Y} (z_1,\dotsc,z_m)}
\end{multline*}
Combining this with Propositions \ref{p.g.f.} and \ref{multi p.g.f.} gives:
\begin{thm}\label{moments}
The conditional factorial moments of $\bold X$ given that $Y_1=k_1,\dotsc, Y_m=k_m$ are:
\begin{multline*}
\EE(X_1^{(s_1)}\dotsb X_d^{(s_d)} \mid \bold Y)  \\
=\frac{[z_1^{k_1}\dotsb z_m^{k_m}] \left( \frac{\del^{s_1}}{\del t_1^{s_1}} \dotsb \frac{\del^{s_d}}{\del t_d^{s_d}} G_{\bold X}(t_1\cdot \Pi_{i=1}^m z_i^{a_{i1}}, \dotsc, t_d\cdot \Pi_{i=1}^m z_i^{a_{id}} ) \right) \big\vert_{t_1=\dotsb=t_d=1} }{[z_1^{k_1}\dotsb z_m^{k_m}] G_{\bold X}(\Pi_{i=1}^m z_i^{a_{i1}}, \dotsc, \Pi_{i=1}^m z_i^{a_{id}} )}
\end{multline*}
\end{thm}
\begin{exmp}
If again $X_1,\dotsc, X_d$ are independent Poisson random variables (\emph{cf.}~\cite{SZ}), $X_r\sim Poisson(\lambda_r)$, $1\leq r\leq n$, then:
\begin{multline*}
\frac{\del^{s_1}}{\del t_1^{s_1}} \dotsb \frac{\del^{s_d}}{\del t_d^{s_d}} G_{\bold X}(t_1\cdot \Pi_{i=1}^m z_i^{a_{i1}}, \dotsc, t_d\cdot \Pi_{i=1}^m z_i^{a_{id}}) \\
= \frac{\del^{s_1}}{\del t_1^{s_1}} \dotsb \frac{\del^{s_d}}{\del t_d^{s_d}} \exp\left( \sum_{r=1}^d \lambda_r \Bigl( t_r \cdot \prod_{i=1}^m z_i^{a_{ir}} -1 \Bigr) \right) \\
=\left(\prod_{r=1}^d \left( \lambda_r \prod_{i=1}^m z_i^{a_{ir}} \right)^{s_r} \right)\exp\left( \sum_{r=1}^d \lambda_r \Bigl( t_r \cdot \prod_{i=1}^m z_i^{a_{ir}} -1 \Bigr) \right)
\end{multline*}
So that:
\begin{multline*}
[z_1^{k_1}\dotsb z_m^{k_m}] \left( \frac{\del^{s_1}}{\del t_1^{s_1}} \dotsb \frac{\del^{s_d}}{\del t_d^{s_d}} G_{\bold X}(t_1\cdot \Pi_{i=1}^m z_i^{a_{i1}}, \dotsc, t_d\cdot \Pi_{i=1}^m z_i^{a_{id}} ) \right) \big\vert_{t_1=\dotsb=t_d=1} \\
=\prod_{r=1}^d \lambda_r^{s_r} \cdot [z_1^{k_1}\dotsb z_m^{k_m}]  \Biggl\{\prod_{i=1}^m z_i^{(\sum_{r=1}^d a_{ir} s_r)} \cdot \exp\Bigl( \sum_{r=1}^d \lambda_r \Bigl(\prod_{i=1}^m z_i^{a_{ir}} -1 \Bigr) \Bigr) \Biggr\} \\
=\prod_{r=1}^d \lambda_r^{s_r} \cdot [z_1^{k_1-\sum_{r=1}^d a_{1r} s_r}\dotsb z_m^{k_m-\sum_{r=1}^d a_{dr}s_r}] G_{\bold Y}(z_1,\dotsc, z_m)
\end{multline*}
Thus:
\begin{multline*}
\EE(X_1^{(s_1)}\dotsb X_d^{(s_d)} \mid \bold Y)  \\
=\prod_{r=1}^d \lambda_r^{s_r} \cdot  \frac{[z_1^{k_1-\sum_{r=1}^d a_{1r} s_r}\dotsb z_m^{k_m-\sum_{r=1}^d a_{dr}s_r}] G_{\bold Y}(z_1,\dotsc, z_m) }{[z_1^{k_1}\dotsb z_m^{k_m}] G_{\bold Y}(z_1,\dotsc, z_m)}
\end{multline*}
whenever $k_i-\sum_{r=1}^d a_{ir} s_r \ge 0$ for all $i$ and $0$ otherwise.  For computations on explicit matrices $\bold A$, Sontag and Zeilberger developed a Maple package utilizing Wilf-Zeilberger Theory to obtain moments and recurrences on the distribution~\cite{SZ}.
\end{exmp}
\begin{exmp}
When $X_1,\dotsc, X_d$ have a $d$-variate multinomial distribution~\cite[pp. 31-92]{JKB} for some $N\in\NN$ and $0\leq p_1,\dotsc, p_d \leq 1$ where $\sum_{i=1}^d p_i=1$, the multivariate generating function is given:
\begin{multline*}
G_{\bold X}(t_1,\dotsc,t_d)
= \sum_{ \substack{(j_1,\dotsc, j_d) \in\NN^{d} \\ j_1+\dotsb + j_d = N}}  \binom{N}{j_1, \dotsc, j_d} p_1^{j_1} \dotsb p_d^{j_d}\cdot t_1^{j_1}\dotsb t_d^{j_d} \\
= (p_1t_1+\dotsb + p_dt_d)^N
\end{multline*}
It follows that:
\begin{multline*}
\frac{\del^{s_1}}{\del t_1^{s_1}} \dotsb \frac{\del^{s_d}}{\del t_d^{s_d}} G_{\bold X}(t_1\cdot \Pi_{i=1}^m z_i^{a_{i1}}, \dotsc, t_d\cdot \Pi_{i=1}^m z_i^{a_{id}}) \\
= \frac{\del^{s_1}}{\del t_1^{s_1}} \dotsb \frac{\del^{s_d}}{\del t_d^{s_d}}  (p_1t_1\Pi_{i=1}^m z_i^{a_{i1}}+\dotsb + p_dt_d \Pi_{i=1}^m z_i^{a_{id}})^N \\
=\frac{N!\prod_{r=1}^d \left( p_r \prod_{i=1}^m z_i^{a_{ir}} \right)^{s_r}}{(N-\sum_{r=1}^d s_r)!} \left(\sum_{r=1}^d p_rt_r\Pi_{i=1}^m z_i^{a_{ir}} \right)^{N-\sum_{r=1}^d s_r}
\end{multline*}
Whence:
\begin{multline*} \tag\dag
[z_1^{k_1}\dotsb z_m^{k_m}] \left( \tfrac{\del^{s_1}}{\del t_1^{s_1}} \dotsb \tfrac{\del^{s_d}}{\del t_d^{s_d}} G_{\bold X}(t_1\cdot \Pi_{i=1}^m z_i^{a_{i1}}, \dotsc, t_d\cdot \Pi_{i=1}^m z_i^{a_{id}} ) \right) \big\vert_{t_1=\dotsb=t_d=1} \\
=\tfrac{N!\prod_{r=1}^d p_r^{s_r}}{(N-\sum_{r=1}^d s_r)!}\cdot [z_1^{k_1}\dotsb z_m^{k_m}] \Bigl\{ \left(\Pi_{i=1}^m z_i^{\sum_{r=1}^d a_{ir}s_r}\right) (\sum_{r=1}^d p_r\Pi_{i=1}^m z_i^{a_{ir}})^{N-\sum_{r=1}^d s_r}\Bigr\} \\
=\tfrac{N!\prod_{r=1}^d p_r^{s_r}}{(N-\sum_{r=1}^d s_r)!}\cdot [z_1^{k_1-\sum_{r=1}^d a_{1r} s_r}\dotsb z_m^{k_m-\sum_{r=1}^d a_{dr}s_r}] \Bigl\{ (\sum_{r=1}^d p_r\Pi_{i=1}^m z_i^{a_{ir}})^{N-\sum_{r=1}^d s_r}\Bigr\}
\end{multline*}
The multinomial theorem permits the re-writing of the bracketed expression in (\dag) above:
\begin{multline*} \tag{\ddag}
(\sum_{r=1}^d p_r\Pi_{i=1}^m z_i^{a_{ir}})^{N-\sum_{r=1}^d s_r} \\
= \sum_{\substack{(j_1,\dotsc, j_d) \in\NN^{d} \\ j_1+\dotsb + j_d = N-\sum_{r=1}^d s_r}} \binom{N-\sum_{r=1}^d s_r}{j_1,\dotsc, j_d} \left( \Pi_{r=1}^d p_r^{j_r} \right) \left(\Pi_{i=1}^m z_i^{\sum_{r=1}^d a_{ir}j_r} \right)
\end{multline*}
So that the coefficient of $z_1^{k_1-\sum_{r=1}^d a_{1r} s_r}\dotsb z_m^{k_m-\sum_{r=1}^d a_{dr}s_r}$ in (\ddag) is:
\[
\sum_{\substack{(j_1,\dotsc, j_d) \in\NN^{d} \\ j_1+\dotsb + j_d = N-\sum_{r=1}^d s_r \\ (k_1,\dotsc,k_m)^T=\bold A(j_1+s_1,\dotsc,j_d+s_d)^T}}
\binom{N-\sum_{r=1}^d s_r}{j_1,\dotsc, j_d} \left( \Pi_{r=1}^d p_r^{j_r} \right)
\]

Dividing the above by $[z_1^{k_1}\dotsb z_m^{k_m}] G_{\bold Y}(z_1,\dotsc, z_m)$ and multiplying by $\tfrac{N!\prod_{r=1}^d p_r^{s_r}}{(N-\sum_{r=1}^d s_r)!}$ will then yield the desired conditional mixed moment.
\end{exmp}

\section*{Acknowledgements}
This paper would not have been possible without the support and guidance of Dr. Eduardo Sontag.  Thanks are also due to Dr. Doron Zeilberger.  Supported in part by grant AFOSR FA9550-11-1-0247.


\bibliographystyle{amsalpha}
\bibliography{LinrTransMultvGenFn}


\end{document}